\documentclass[11pt,english,a4paper]{amsart}
\pdfoutput=1

\usepackage[foot]{amsaddr}
\usepackage{a4wide}

\usepackage[T1]{fontenc}
\usepackage[utf8]{inputenc}
\usepackage{hyperref}
\usepackage{amsmath}
\usepackage{amssymb}
\usepackage{amsfonts}
\usepackage{graphicx}
\usepackage{color}
\usepackage{colonequals}
\usepackage{MnSymbol}
\usepackage{esint}
\usepackage{tikz}

\definecolor{grey}{rgb}{0.5,0.5,0.5}

\newcommand{\Prod}[1]{\langle #1 \rangle}
\newcommand{\op}[1]{\operatorname{#1}}
\newcommand{\R}{\mathbb{R}}

\newcommand{\st}{\mathrel\vert}
\newcommand{\alphamin}{\alpha_{\min}}
\newcommand{\alphamax}{\alpha_{\max}}
\newcommand{\tildealpha}{\tilde{\alpha}}
\newcommand{\tildealphamin}{\tilde{\alpha}_{\min}}
\newcommand{\tildealphamax}{\tilde{\alpha}_{\max}}
\newcommand{\sigmamax}{\|\sigma\|_{L^\infty(\Omega)}}
\newcommand{\tildesigma}{\tilde{\sigma}}
\newcommand{\tildesigmamax}{\|\tildesigma\|_{L^\infty(\Omega)}}
\newcommand{\dataweight}{\delta^{-1}}
\newcommand{\invdataweight}{\delta}
\newcommand{\Eapprox}{\mathbb{E}_{\op{approx}}}
\newcommand{\Epde}{\mathbb{E}_{\op{pde}}}
\newcommand{\Edata}{\mathbb{E}_{\op{data}}}
\newcommand{\FINT}{{\textstyle\fint}}

\newcommand{\FUNDING}{
Funded by the Deutsche Forschungsgemeinschaft (DFG, German Research Foundation)
under Germany´s Excellence Strategy – The Berlin Mathematics Research Center MATH+
(EXC-2046/1, project ID: 390685689).
Project EF3-4 ``Physics-regularized learning''.
}

\newcommand{\ABSTRACT}{
  In this work we consider the regularization of a supervised
  learning problem by partial differential equations (PDEs)
  and derive error bounds for the obtained approximation
  in terms of a PDE error term and a data error term.
  Assuming that the target function satisfies an unknown PDE,
  the PDE error term quantifies how well this PDE is approximated
  by the auxiliary PDE used for regularization.
  It is shown that this error term decreases if more data
  is provided.
  The data error term quantifies the accuracy of the given data.
  Furthermore, the PDE-regularized learning problem is discretized
  by generalized Galerkin discretizations solving
  the associated minimization problem in subsets of the
  infinite dimensional functions space, which are not
  necessarily subspaces.
  For such discretizations an error bound in terms of
  the PDE error, the data error, and a best approximation
  error is derived.
}

\newtheorem{theorem}{Theorem}
\newtheorem{lemma}{Lemma}
\newtheorem{corollary}{Corollary}
\newtheorem{proposition}{Proposition}

\newtheorem{definition}{Definition}

\newtheorem{remark}{Remark}

\begin{document}

\title{Error bounds for PDE-regularized learning}

\author[Gr\"aser]{Carsten Gr\"aser$^1$}
\address{
  $^1$Institut f\"ur Mathematik,
  Freie Universit\"at Berlin,
  14195 Berlin, Germany
}
\email{graeser@mi.fu-berlin.de}

\author[Alathur Srinivasan]{Prem Anand Alathur Srinivasan$^2$}
\address{
  $^2$Institut f\"ur Mathematik,
  Freie Universit\"at Berlin,
  14195 Berlin, Germany
}
\email{premanand@zedat.fu-berlin.de}

\thanks{\FUNDING}

\begin{abstract}
  \ABSTRACT
\end{abstract}

\date{}

\maketitle

\section{Introduction}

The problem of learning an unknown function $u:\Omega \to \R$
is one of the central problems in the field of machine learning.
A classical ansatz is to minimize the risk functional
\begin{align*}
  \tilde{u} \mapsto |\Omega|^{-1} \int_\Omega L(\tilde{u}(x),u(x))\,dx
\end{align*}
in some ansatz set $V_h$ for a loss functional $L : \R^2 \to \R$.
In case of a quadratic loss functional
this leads to an $L^2(\Omega)$-best-approximation problem
for $u$ in $V_h \subset L^2(\Omega)$.

Solving this problem in general requires complete knowledge
of $u\in L^2(\Omega)$.
In practice the available data on $u$ is often incomplete.
A classical example of incomplete data is a finite set
of known point data $(p_1,u(p_1)),\dots,(p_m,u(p_m))$
for a continuous function $u$.
In this case, instead of minimizing the risk, one often
considers the problem of \emph{empirical risk minimization}
given by
\begin{align}
  \label{eq:general_learning_problem}
  \tilde{u}_h = \op{argmin}_{v \in V_h} \frac{1}{m}\sum_{i=1}^m |v(p_i)-u(p_i)|^2.
\end{align}
Here, the weighted sum can be viewed as Monte-Carlo approximation
of the $L^2(\Omega)$-norm using the sample set $\{p_1,\dots,p_m\}$.

Recently a lot of attention has been paid to learning with
shallow or deep neural networks. In this case
$V_h \subset L^2(\Omega)$ is given as the range
$V_h = \Phi(\R^N)$
of a nonlinear map $\Phi : \R^N \to L^2(\Omega)$
over a finite dimensional parameter domain $\R^N$, such that
the empirical risk minimization problem in $V_h$
can be written as
\begin{align}
  \label{eq:nn_fitting}
  W^* = \op{argmin}_{W\in \R^N} \frac{1}{m}\sum_{i=1}^m |\Phi(W)(p_i)-u(p_i)|^2
\end{align}
with $\tilde{u}_h = \Phi(W^*)$.
The minimizer $W^*$ is often approximated using stochastic
gradient-type algorithms.
While this approach has been used with reasonable success
in practice, its theoretical understanding is still in its
infancy. The central question is, if the total error, i.e.
the difference $\tilde{u}_h-u$ of the computed approximation
and the target function can be controlled.
When analyzing this error, several aspects have to be taken
into account.

\textbf{Learning:} Here the question is if we can control
the error made by the algebraic solution method, e.g.
stochastic gradient descent.
If the algebraic error $W-W^*$ of a computed approximate
parameter set $W$ can be controlled, local Lipschitz continuity
of the parametrization $\Phi$ allows to control
the induced error $\Phi(W)-\Phi(W^*)$ of the associated functions.
A major obstacle for an error analysis is, that the
learning problem~\eqref{eq:nn_fitting} is in general
nonconvex and may have multiple global or local minimizers,
saddle points, or plateaus.
One direction to target this problem is to characterize
network architectures, where local optimality implies global
optimality or at least plateaus~(see, e.g., \cite{HaeffeleVidal2017},~\cite{VidalEtAl2017}
and the references therein).
Another direction is to interpret a (stochastic) gradient flow
for shallow neural networks as evolution of an interacting
particle system which allows to relate long time limits
(i.e. stationary points) and the many particle limit~\cite{RotskoffVandenEijnden2018}.
In the present paper we do not consider this aspect
and concentrate on the analysis of global minimizers.

\textbf{Expressivity:} This centers around the question
of how well $u$ can be approximated in $V_h$?
Starting from early results on universal approximation properties
of neural networks (see, e.g. \cite{Funahashi1989, Pinkus1999})
there was significant recent progress on deriving
best approximation error and expressivity bounds for
deep neural networks
\cite{
  BolcskeiEtAl2019,
  BurgerNeubauer2001,
  GribovalEtAl2019,
  GrohsEtAl2018,
  GuehringKuyniokPetersen2019,
  PerekrestenkoEtAl2018,
  PetersenVoigtlaender2018,
  ShenYangZhang2019,
  Yarotsky2017}.
An additional question is, to which extend it is possible
to realize the theoretically derived best approximation error
by solving the learning problem.
Despite its importance this question is largely unexplored
which is rooted in the fact that it can hardly be answered
due to ill-posedness in case of incomplete data.
In the present paper we will address this question
by regularizing the problem and proving discretization error
bounds for the minimizer $\tilde{u}_h$ in terms of the best
approximation error for $u$ in $V_h$.

\textbf{Generalization:}
This addresses the question if the trained function generalizes
from the training data to other input data. In mathematical terms
this leads to the question of how well the computed function
$\tilde{u}_h$ can approximate $u$ on $\Omega$
given incomplete data at some points
$\{p_1,\dots,p_m\} \subsetneqq \Omega$.
This question is often discussed in a statistical setting considering
the input data as randomly drawn samples.
In the present paper we will take a deterministic perspective
and are interested in error bounds for $\tilde{u}_h-u$ in terms
of the amount of provided data. Again such bounds can in general
not be derived due to ill-posedness, which we will address by a regularization
of the problem.

The fact that it is hard to derive error bounds%
---even for global minimizers of the learning problem---%
is deeply related to the fact that problem~\eqref{eq:general_learning_problem}
which only incorporates incomplete data
is in general ill-posed leading to severe artifacts.
For example, if the ansatz set $V_h$ is
'large' in comparison to the available amount of data,
one observes so called \emph{overfitting},
where $\tilde{u}$ matches $u$ nicely in the points $p_i$ but
fails to provide a reasonable approximation in other parts of $\Omega$.
This
is in fact a consequence of ill-posedness
\interfootnotelinepenalty=10000
\footnote{%
  This can be easily
  explained in the case of approximation by polynomials
  $V_h = \mathcal{P}_{n}$ of degree $n$.
  If $n=m-1$, then the resulting $\tilde{u}$ is exactly
  the interpolation polynomial which in general exhibits
  severe over- and undershoots. If $n\geq m$ then there is no
  unique solution and one can add an oscillatory
  polynomial of degree $n$ with arbitrary amplitude
  to the interpolation polynomial leading to uncontrollably large errors.
}.
A common technique to reduce such artifacts is to introduce regularization terms for $W$.
However, it is unclear how the influence of such regularization
on the error $\Phi(W^*) -u = \tilde{u}_h-u$ with respect to the $L^2(\Omega)$-norm
could be analyzed.
Another obstacle in deriving error bounds for $\tilde{u}_h-u$
is, that it is unclear,
how to understand~\eqref{eq:general_learning_problem} as discretization
of a continuous problem and thus how such a discretization
and the influence of incomplete data could be analyzed.
In the present paper we target these questions by
introducing regularizations
of the learning problem
by partial differential equations (PDEs).
This is based on the assumption that additional
knowledge on the process generating the data or at least
on the smoothness of $u$ is available.
Using such regularizations we derive a framework
which allows to prove error bounds for
$\tilde{u}_h-u$ that quantify the effect of incomplete
data and relate the discretization error in nonlinear
ansatz spaces (like neural networks) to their approximation properties.

Due to their approximation power, neural networks
have already been proposed for the solution of partial
different equations (PDEs).
In~\cite{EYu2018} the authors introduce the \emph{deep Ritz} method
which approximates the solution of a variational formulation
of a PDE using discretization by deep neural networks.
In contrast to~\eqref{eq:nn_fitting} this approach
minimizes the Dirichlet energy associated to the PDE
in the nonlinear neural network space using
a simple penalty approach for essential boundary data.
As a variant, it was proposed to minimize the consistent
penalty formulation from Nitsche's method using a
heuristic penalty parameter~\cite{LiaoMing2019}.
While error bounds for both methods are unknown so far
a convergence result based on $\Gamma$-convergence was recently derived~\cite{MuellerZeinhofer2019}.
A different approach was taken in~\cite{RaissiPerdikarisKarniadakis2017a},
where the authors introduce so called
\emph{physics informed neural networks} (PINNs)
which are trained using a collocation least squares functional.
This ansatz is extended in~\cite{RaissiPerdikarisKarniadakis2017b}
where, additionally to the PDE, point data of the target function
is incorporated.
Replacing collocation by a least squares Petrov--Galerkin
ansatz leads to the \emph{variational PINN}
approach considered in~\cite{KharazmiZhangKarniadakis2019}.
It has also been highlighted that a neural network ansatz
is especially promising for high-dimensional
PDEs~\cite{EYu2018,HanJentzenE2018,GrohsEtAl2018,LiaoMing2019}.
Other uses of neural networks in the context of PDEs e.g.
include reduced basis methods for parametrized problems~\cite{KutyniokEtAl2019}.

\textbf{Our contribution:}
In the present paper
we will in general assume that $u$ solves an elliptic PDE
which is not known exactly.
Since we cannot use the unknown exact PDE,
an inexact auxiliary PDE is used to regularize the problem.
Furthermore, to make the problem
well-posed in a Lebesgue- and Sobolev-space setting, we first replace
the point data $u(p_i)$ by local averages $\FINT_{B_i}u \approx f(p_i)$
on sets $B_i \subset \Omega$ with $p_i \in B_i$
leading to the regularized learning problem
\begin{align*}
  \tilde{u} = \op{argmin}_{v \in V}
    \frac12 \sum_{i=1}^m |B_i| \left| \FINT_{B_i} v- \FINT_{B_i} u\right|^2
    + \invdataweight \left(\frac12 \tilde{a}(v,v) - \tilde{\ell}(v) \right).
\end{align*}
Here, $\frac12 \tilde{a}(v,v) - \tilde{\ell}(v)$ is the
Dirichlet energy associated to the elliptic auxiliary PDE and
$\invdataweight>0$ is a regularization parameter that balances
the data and PDE term.
The main result of the paper is an error bound
of the form
\begin{align*}
  \|\tilde{u} - u\|_{L^2(\Omega)}
    \leq CR^2 \Epde.
\end{align*}
Here $R$ is a constants that can be decreased
by adding more data and $\Epde$ quantifies the error
induced by using an inexact auxiliary PDE.
Hence the accuracy of $\tilde{u}$ can be improved
by either providing more data or by improving the exactness
of the auxiliary PDE.
In a second step we treat the case of given point values $u(p_i)$
by considering $u(p_i)$ as an inexact variant of
$\FINT_{B_i} u$.
Assuming $\lambda$-Hölder-continuity of $u$ we
can control the additional error by an error bound
\begin{align*}
  \|\tilde{u} - u\|_{L^2(\Omega)}
    \leq CR^2 \Epde + C r^\lambda
\end{align*}
where $r$ can again be decreased by adding more data.

Finally, using a nonlinear C\'ea-Lemma, we derive a bound
\begin{align*}
  \|\tilde{u}_h - u\|_{L^2(\Omega)}
    \leq CR^2 \Epde + C r^\lambda
    + C\inf_{v\in V_h} \Bigl(R \|\nabla(v-u)\|_{L^2(\Omega)} + \|v-u\|_{L^2(\Omega)} \Bigr)
\end{align*}
for the case of a generalized Galerkin discretization where
$\tilde{u}_h$ is computed by minimizing in a
subset $V_h$ of $V$.
A variant for inexact minimizers is also presented.
Since the result allows for non-subspace subsets $V_h \subset V$,
it is in principle applicable to nonlinear approximation
schemes like neural networks.
Inserting known approximation error bounds
(e.g. from~\cite{GuehringKuyniokPetersen2019})
on the right hand side,
this leads to a discretization error bound
for neural network discretizations under the assumption
that a global minimizer can be computed, or, that
the algebraic energy error can be controlled.
While the latter can in general not be guaranteed,
the derived results open a new perspective for
an error analysis of neural networks.
For example, the same arguments can be used
to provide a-priori error bounds for
the deep Ritz method~\cite{EYu2018} for Neumann problems
with exact (or controlled inexact) global minimizers
which allows to quantify recent convergence results~\cite{MuellerZeinhofer2019}.

The paper is organized as follows:
First the PDE-regularized learning problem
is introduced and its well-posedness is discussed
in Section~\ref{sec:problem}.
Then an $L^2(\Omega)$ error bound for $\tilde{u}-u$
is derived in Section~\ref{sec:error_analysis}
in the infinite dimensional case.
This section also discussed the quasi-optimality
of the derived error bound for a pure data fitting
problem without a-priori knowledge on the PDE.
The generalized Galerkin discretization in subsets $V_h$
is introduced and analyzed in Section~\ref{sec:discretization}.
Finally, the theoretical findings are illustrated
by numerical experiments with finite element and
neural network discretizations in Section~\ref{sec:numerical_results}.

\section{Problem setting}
\label{sec:problem}

\subsection{Exact and inexact auxiliary PDEs}
We are interested in approximating a function
$u:\Omega \to \R$ on a bounded domain $\Omega \subset \R^d$
with Lipschitz boundary.
Throughout the paper we make the assumption that
$u$ solves an elliptic partial differential equation (PDE)
given in terms of a variational equation
\begin{align}
  \label{eq:exact_problem}
  u\in V:
    \qquad
    a(u,v) = \ell(v)
    \qquad \forall v \in V
\end{align}
for a closed subspace $V \subset H^1(\Omega)$ with $H_0^1(\Omega) \subset V$,
a symmetric bilinear form
\begin{align*}
  a(w,v)
    \colonequals \int_\Omega \alpha(x) \nabla w(x) \cdot \nabla v(x)
    + \sigma(x) w(x) v(x) \,dx
\end{align*}
with uniformly bounded coefficient functions
$\alpha :\Omega \to \R$ and $\sigma :\Omega \to \R$,
\begin{align*}
  0 < \alphamin \leq \alpha(x) \leq \alphamax < \infty, \qquad
  0 \leq \sigma(x) \leq \sigmamax < \infty
\end{align*}
and $\ell \in V^*$ given by
\begin{align*}
  \ell(v) \colonequals \int_\Omega f(x)\, dx
\end{align*}
for some $f\in L^2(\Omega)$.
The subspace $V$ is chosen such that $a(\cdot,\cdot)$
is coercive on $V$. Then \eqref{eq:exact_problem}
has a unique solution $u\in V$
by the Lax--Milgram theorem.

The basic assumption we will make is, that the PDE is not known
exactly, that is, the exact functions $\alpha,\sigma,\ell$ are unknown.
Instead we will consider an inexact auxiliary PDE induced
by a guess for $\alpha$, $\sigma$, and $f$.
The auxiliary PDE is given in terms a bilinear form
\begin{align*}
  \tilde{a}(w,v)
    \colonequals \int_\Omega \tildealpha(x) \nabla w(x) \cdot \nabla v(x)
    + \tildesigma(x) w(x) v(x) \,dx
\end{align*}
with uniformly bounded coefficient functions
$\tildealpha :\Omega \to \R$ and $\tildesigma :\Omega \to \R$,
\begin{align*}
  0 < \tildealphamin \leq \tildealpha(x) \leq \tildealphamax < \infty, \qquad
  0 \leq \tildesigma(x) \leq \tildesigmamax < \infty
\end{align*}
and a functional $\tilde{\ell} \in V^*$ given by
\begin{align*}
  \tilde{\ell}(v) \colonequals \int_\Omega \tilde{f}(x)\, dx
\end{align*}
for some $\tilde{f}\in L^2(\Omega)$.

\subsection{PDE-regularized learning problem}

To compute $\tilde{u}$ we will not just solve $\tilde{a}(\tilde{u},\cdot)-\tilde{\ell}=0$
but combine this PDE with the given (possibly inexact)
local data on $u$.
To this end we assume that the data is given in terms of local
average values of $u$ on open, nonempty subsets $B_i \subset \Omega$
for $i=1,\dots,m$.
We will also assume that the sets $B_i$ can be extended
to form a covering of $\Omega$ in the sense that for each
$B_i$ there is a convex set $K_i$ with Lipschitz boundary
and $B_i \subset K_i$ such that
\begin{align*}
  \Omega \subset \bigcup_{i=1}^m \overline{K_i}.
\end{align*}
The maximal overlap and the maximal diameter
of the families $(B_i)$ and $(K_i)$ given by
\begin{align*}
  M \colonequals \max_{x \in \Omega} |\{ i \st x \in K_i\}|,
  \qquad
  r \colonequals \max_{i=1,\dots,m} \op{diam}(B_i),
  \qquad
  R \colonequals \max_{i=1,\dots,m} \op{diam}(K_i)
\end{align*}
will be used to quantify errors later on.

In the following we will use the notation
$\FINT_U v = |U|^{-1}\int_U v(x)\, dx$ for the average
of $v$ over a bounded set $U$.
Using a regularization parameter $\invdataweight>0$ we define
define the augmented auxiliary forms
\begin{align*}
  \tilde{c}(w,v) \colonequals \tilde{a}(w,v) + \dataweight b(w,v),
  \qquad
  \tilde{r}(v) \colonequals \tilde{\ell}(v) + \dataweight \widetilde{b(u,v)}.
\end{align*}
with the bilinear form $b(\cdot,\cdot)$
and the possibly inexact data term $\widetilde{b(u,\cdot)}\approx b(u,\cdot)$ given by
\begin{align}
  \label{eq:data_term}
  b(w,v) \colonequals \sum_{i=1}^m |B_i| \FINT_{B_i} w \FINT_{B_i} v,
  \qquad
  \widetilde{b(u,v)} \colonequals \sum_{i=1}^m b_i|B_i|\FINT_{B_i} v.
\end{align}
This data term can be viewed as an approximation of $b(u,v)$
which becomes exact for $b_i = \FINT_{B_i} u$.
Using this notation we introduce
the PDE-regularized learning problem
\begin{align}
  \label{eq:augmented_auxiliary_problem_min}
  \tilde{u} = \op{argmin}_{v\in V} \tilde{J}(v),
\end{align}
for the functional
\begin{align}
  \label{eq:functional}
  \tilde{J}(v) \colonequals \frac12 \tilde{c}(v,v) - \tilde{r}(v)
    = \frac12 \sum_{i=1}^m |B_i| \left| \FINT_{B_i} v- b_i\right|^2
    + \delta\Bigl(\frac12 \tilde{a}(v,v) - \tilde{l}(v)\Bigr) + \op{const}.
\end{align}
Using standard arguments we see, that this quadratic
minimization problem is equivalent to
\begin{align}
  \label{eq:augmented_auxiliary_problem}
  \tilde{u} \in V:
  \qquad
  \tilde{c}(\tilde{u},v) = \tilde{r}(v)
  \qquad
  \forall v \in V.
\end{align}
Analogously to \eqref{eq:augmented_auxiliary_problem}
we can define the augmented exact problem
\begin{align}
  \label{eq:augmented_exact_problem}
  u\in V:
  \qquad
  c(u,v) = r(v)
  \qquad
  \forall v \in V
\end{align}
for the exact augmented forms
\begin{align*}
  c(w,v) \colonequals a(w,v) + \dataweight b(w,v),
  \qquad
  r(v) \colonequals \ell(v) + \dataweight b(u,v).
\end{align*}
It is straight forward to show that \eqref{eq:augmented_exact_problem}
is equivalent to \eqref{eq:exact_problem}.

\subsection{Well-posedness by PDE-based regularization}
We now discuss well-posedness of the PDE-regularized learning
problem.
For convenience we will denote the semi-norm
$v \mapsto d(v,v)^{\frac12}$ induced by a symmetric,
positive semi-definite, bilinear form $d(\cdot,\cdot)$
by $\|\cdot\|_d = d(\cdot,\cdot)^{\frac12}$.
Using this notation it is clear that
\begin{align*}
  \|\cdot\|^2_{\tilde{c}} &= \|\cdot\|^2_{\tilde{a}} + \dataweight \|\cdot\|^2_b, &
  \|\cdot\|^2_c &= \|\cdot\|^2_a + \dataweight \|\cdot\|^2_b.
\end{align*}
The following lemma shows that the weighting of the terms in
$b(\cdot,\cdot)$ is natural in the sense that it guarantee that
$\|\cdot\|_b$ scales like $\|\cdot\|_{L^2(\Omega)}$.

\begin{lemma}
  \label{lem:data_form_L2_cont}
  The bilinear form $b(\cdot,\cdot)$ is $L^2(\Omega)$-continuous with
  \begin{align*}
    b(w,v) \leq M \|w\|_{L^2(\Omega)} \|v\|_{L^2(\Omega)}, \qquad
    \|v\|_b^2 \leq M \|v\|^2_{L^2(\Omega)}, \qquad
    \forall w,v \in L^2(\Omega).
  \end{align*}
\end{lemma} 
\begin{proof}
  Let $v,w \in L^2(\Omega)$. We first note that
  the Cauchy--Schwarz inequality gives
  \begin{align}
    \label{eq:av_l2_cont}
    \FINT_{B_i} v \leq |B_i|^{-1}\|v\|_{L^2(B_i)} \|1\|_{L^2(B_i)} = |B_i|^{-1/2}\|v\|_{L^2(B_i)}.
  \end{align}
  Using this estimate and the Cauchy--Schwarz inequality in $\R^m$ we get
  \begin{align*}
    b(w,v)
      \leq \sum_{i=1}^m \|w\|_{L^2(B_i)}\|v\|_{L^2(B_i)}
      \leq \left(\sum_{i=1}^m \|w\|_{L^2(B_i)}^2\right)^{1/2}
            \left(\sum_{i=1}^m \|v\|_{L^2(B_i)}^2\right)^{1/2}
      \leq M\|w\|_{L^2(\Omega)}\|v\|_{L^2(\Omega)}.
  \end{align*}
\end{proof}

Despite this upper bound, a pure data fitting problem
\begin{align*}
  \op{min}_{v \in L^2(\Omega)}
    \frac12 \sum_{i=1}^m |B_i| \left| \FINT_{B_i} v- b_i\right|^2
  = \op{min}_{v \in L^2(\Omega)}
    \frac12 \|v\|_b^2 - \widetilde{b(u,v)} + \op{const}
\end{align*}
is in general ill-posed, since $b(\cdot,\cdot)$ is not coercive,
or, equivalently, $\|\cdot\|_b$ cannot be bounded from below
by $\|\cdot\|_{L^2(\Omega)}$.
Due to the finite rank $m$ of $b(\cdot,\cdot)$, the
same is true whenever minimization is considered
in a space with dimension larger then $m$.
Thanks to the PDE-regularization,
the situation is different for the PDE-regularized
problem~\eqref{eq:augmented_auxiliary_problem_min}.

\begin{proposition}
  \label{prop:augmented_auxiliary_existence}
  The PDE-regularized problem~\eqref{eq:augmented_auxiliary_problem_min}
  or, equivalently,~\eqref{eq:augmented_auxiliary_problem}
  has a unique solution $\tilde{u} \in V$
  which depends Lipschitz-continuously on
  the provided data $b_1,\dots,b_m$.
\end{proposition}
\begin{proof}
  By the lower bound on $\tildealpha$ we obtain
  \begin{align*}
    \|\nabla v\|_{L^2(\Omega)}^2
      \leq \tildealphamin^{-1} \|v\|^2_{\tilde{c}}
      \qquad \forall v \in V.
  \end{align*}
  Furthermore $b(\cdot,\cdot)$ is coercive on the
  space of constant functions.
  Hence coercivity of $\tilde{c}$ on $H^1(\Omega)$
  follows from the Poincar\'{e} inequality
  (see \cite[Proposition~2]{Graeser2015}).
  Furthermore the upper bounds on $\tildealpha$ and $\tildesigma$
  and Lemma~\ref{lem:data_form_L2_cont}
  imply continuity and thus ellipticity of $\tilde{c}(\cdot,\cdot)$
  on $V$.
  Finally, we get continuity of $\widetilde{b(u,\cdot)}$
  and thus $\tilde{r}(\cdot)$ similar to the proof of
  Lemma~\ref{lem:data_form_L2_cont} such that
  the Lax--Milgram theorem guarantees existence of a unique
  solution $\tilde{u} \in V$ of~\eqref{eq:augmented_auxiliary_problem}.
  Lipschitz continuous dependency on $b_1,\dots,b_m$
  follows from standard arguments
  for linear elliptic problems.
\end{proof}

\section{Error bounds for PDE-regularized learning}
\label{sec:error_analysis}

\subsection{Localized Poincar\'{e} inequality and improved $L^2(\Omega)$-coercivity}

A central ingredient in the error estimates shown later
is the $L^2(\Omega)$-coercivity of $\tilde{c}(\cdot,\cdot)$.
In the proof of Proposition~\ref{prop:augmented_auxiliary_existence}
we have seen, that $\tilde{c}(\cdot,\cdot)$
inherits coercivity with respect to the $H^1(\Omega)$-
and $L^2(\Omega)$-norm from $\tilde{a}(\cdot,\cdot)$.
However, the coercivity constants are independent of the data term.
In this section we will show an improved 
$L^2(\Omega)$-coercivity by applying localized Poincar\'e estimates.
First we remind the classical Poincar\'e inequality on convex
domains.

\begin{lemma}
  \label{lemma:poincare_average}
  Let $U \subset \R^d$ be a convex, open, non-empty, bounded domain
  and $v \in H^1(U)$. Then
  \begin{align*}
    \left\|v-\FINT_U v\right\|_{L^2(U)} \leq C_U \|\nabla v\|_{L^2(U)}
  \end{align*}
  with Poincar\'e constant $C_U = \frac{\op{diam}(U)}{\pi}$.
\end{lemma}

In fact the constant $C_U$ given here is the best possible
one depending only on the domain diameter. A proof can be found
in \cite{PayneWeinberger1960}.
As a direct consequence we get the following estimate.

\begin{lemma}
  \label{lemma:poincare_int}
  Let $U \subset \R^d$ be a convex, open, non-empty, bounded domain
  and $v \in H^1(U)$. Then
  \begin{align*}
    \|v\|^2_{L^2(U)}
    \leq C_U^2 \|\nabla v\|^2_{L^2(U)} + \|\FINT_U v\|_{L^2(U)}^2
    = C_U^2 \|\nabla v\|^2_{L^2(U)} + |U|\left(\FINT_U v\right)^2
  \end{align*}
  with Poincar\'e constant $C_U = \frac{\op{diam}(U)}{\pi}$.
\end{lemma}

\begin{proof}
  Using the $L^2(U)$-orthogonality $v-\FINT_U v$ and $\FINT_U v$
  and Lemma~\ref{lemma:poincare_average} we get
  \begin{align*}
    \|v\|_{L^2(U)}^2
      = \left\|v-\FINT_U v\right\|_{L^2(U)}^2
        + \left\|\FINT_U v\right\|_{L^2(U)}^2
    \leq C_U^2 \|\nabla v\|^2_{L^2(U)} + |U|\left(\FINT_U v\right)^2.
  \end{align*}
\end{proof}

It is also possible to get bounds involving just averages
over subsets, at the price of an additional constant.
An abstract prove has been given in \cite{Graeser2015}.
Since we are interested in the resulting constants
we will give an explicit proof here.

\begin{lemma}
  \label{lemma:poincare_int_subset}
  Let $U \subset \R^d$ be a convex, open, non-empty, bounded domain,
  $W \subset U$ with $|W|>0$, $t>0$,
  and $v \in H^1(U)$. Then
  \begin{align*}
    \|v\|^2_{L^2(U)}
      \leq C_U^2 \left(1+(1+t) \frac{|U|}{|W|}\right) \|\nabla v\|^2_{L^2(U)}
      + (1+t^{-1})|U|\left(\FINT_W v\right)^2
  \end{align*}
  with Poincar\'e constant $C_U = \frac{\op{diam}(U)}{\pi}$
  and especially (using $t=1$)
  \begin{align*}
    \|v\|^2_{L^2(U)}
      \leq 3 \frac{|U|}{|W|} \Bigl(C_U^2  \|\nabla v\|^2_{L^2(U)}
      + |W|\left(\FINT_W v\right)^2 \Bigr).
  \end{align*}
\end{lemma}

\begin{proof}
  Applying Lemma~\ref{lemma:poincare_int} on $U$ we get
  \begin{align*}
    \|v\|_{L^2(U)}^2
        \leq C_U^2 \|\nabla v\|^2_{L^2(U)} + \|v_U\|_{L^2(U)}^2
  \end{align*}
  where we used the notation $v_U = \FINT_U v$.
  Utilizing the inequality
  \begin{align*}
    \Prod{a,a}
      = \Prod{a-b,a-b} +2\Prod{b,a-b} + \Prod{b,b}
      \leq (1+t^{-1})\Prod{b,b} + (1+t)\Prod{b-a,b-a}
  \end{align*}
  for the symmetric positive semi-definite bilinear form
  $\Prod{\cdot,\cdot}=\FINT_W(\cdot)\FINT_W(\cdot)$
  we get
  \begin{align*}
    \|v_U\|_{L^2(U)}^2
      = |U|\bigl(\FINT_W v_U \bigr)^2
      &\leq
        |U|\Bigl(
          (1+t^{-1}) \bigl(\FINT_W v \bigr)^2
          + (1+t) \bigl(\FINT_W (v-v_U) \bigr)^2
        \Bigr)\\
      &\leq
        |U|\Bigl(
          (1+t^{-1}) \bigl(\FINT_W v \bigr)^2
          + (1+t) |W|^{-1}\| v-v_U \|_{L^2(W)}^2
        \Bigr)\\
      &\leq
          |U|(1+t^{-1}) \bigl(\FINT_W v \bigr)^2
          + (1+t) \frac{|U|}{|W|}\| v-v_U \|_{L^2(U)}^2.
  \end{align*}
  Finally we get the assertion by using Lemma~\ref{lemma:poincare_average}.
\end{proof}

Note that using the optimal value for $t$ the constant $3$
can be slightly improved to $1+\varphi<3$ with the golden ratio
$\varphi=\frac12(1+\sqrt{5})$.

Next we apply Lemma~\ref{lemma:poincare_int} and
Lemma~\ref{lemma:poincare_int_subset} locally to show
a data dependent global Poincar\'e type estimate.
To unify both estimates we encode the maximal mismatch
of $B_i$ and $K_i$ in terms of the constant
\begin{align*}
  \eta \colonequals
    \begin{cases}
      1
        &\text{if } B_i=K_i \quad \forall i,\\
      3 \max_{i=1,\dots,m} \frac{|K_i|}{|B_i|}.
        &\text{else}.
    \end{cases}  
\end{align*}

\begin{lemma}
  \label{lem:poincare_with_data}
  Using the constants $M,R,\eta$ defined above it holds
  for $v\in H^1(\Omega)$
  \begin{align*}
    \|v\|_{L^2(\Omega)}^2 \leq \eta \left(
      \frac{R^2M}{\pi^2} \|\nabla v\|_{L^2(\Omega)}^2 + b(v,v)
      \right).
  \end{align*}
\end{lemma}

\begin{proof}
  Let $v \in H^1(\Omega)$. Then we also have $v \in H^1(B_i)$ for
  $i=1,\dots,m$.
  Applying either Lemma~\ref{lemma:poincare_int} (if $\eta=1$)
  or Lemma~\ref{lemma:poincare_int_subset} (if $\eta>1$)
  on each $B_i$ we get
  \begin{align*}
    \|v\|_{L^2(\Omega)}^2
      \leq \sum_{i=1}^m \|v\|_{L^2(B_i)}^2
      &\leq \eta\sum_{i=1}^m
        \left(
          \frac{\op{diam}(K_i)^2}{\pi^2} \|\nabla v\|^2_{L^2(K_i)}
          + |B_i|\left(\FINT_{B_i} v\right)^2
        \right)\\
      &\leq \eta\left(
        \sum_{i=1}^m
        \frac{R^2}{\pi^2} \|\nabla v\|^2_{L^2(B_i)}
        + b(v,v)
        \right).
  \end{align*}
  Using the fact that each part of $\Omega$ is a most covered
  $M$-times by the sets $B_i$ we get the assertion.
\end{proof}

The right hand side in the estimate of Lemma~\ref{lem:poincare_with_data}
almost coincides with $\tilde{c}(v,v)$.
The following lemma balances the constants to finally
derive an $L^2(\Omega)$-coercivity of $\tilde{c}(\cdot,\cdot)$.

\begin{lemma}
  \label{lem:l2coercivity}
  For any $v \in H^1(\Omega)$ it holds that
  \begin{align*}
    \|v\|^2_{L^2(\Omega)}
      \leq \Gamma \|v\|_{\tilde{c}}^2,
      \qquad
      \Gamma \colonequals \eta\max\left\{\frac{R^2M}{\pi^2 \tildealphamin}, \invdataweight \right\}.
  \end{align*}
\end{lemma}
\begin{proof}
  Using Lemma~\ref{lem:poincare_with_data},
  positive definiteness of $\|\sqrt{\tildesigma}\cdot\|_{L^2(\Omega)}^2$,
  and $\Gamma \dataweight \geq \eta$ we get
  \begin{align*}
    \|v\|_{L^2(\Omega)}^2
      \leq
        \eta\frac{R^2M}{\pi^2 \tildealphamin} \|v\|_{\tilde{a}}^2 + \eta \|v\|_b^2
      \leq
        \eta\frac{R^2M}{\pi^2 \tildealphamin} \|v\|_{\tilde{a}}^2 + \Gamma\dataweight \|v\|_b^2
      \leq \Gamma \|v\|_{\tilde{c}}^2.
  \end{align*}
\end{proof}

In the following the relation of $\invdataweight$ and the constant from
Lemma~\ref{lem:poincare_with_data} will play a crucial role.
In the simplest case we would have
\begin{align}
  \label{eq:optimal_data_weight}
  \invdataweight = \frac{R^2M}{\pi^2 \tildealphamin},
\end{align}
such that the constant in Lemma~\ref{lem:l2coercivity} reduces to $\Gamma=\eta \invdataweight$.
Since this constant may not be known exactly we quantify the relation
by assuming that
\begin{align}
  \label{eq:weight_choice}
  \invdataweight \in \left[ \theta\frac{R^2M}{\pi^2 \tildealphamin},\theta^{-1}\frac{R^2M}{\pi^2 \tildealphamin} \right]
\end{align}
for $\theta \in (0,1]$.
Note that by $\invdataweight \in (0,\infty)$ such a $\theta=(0,1)$ does always exist.
Using this assumption we obtain the following bounds on $\Gamma$
\begin{align}
  \label{eq:weight_constant_bounds}
  \Gamma \leq \eta \theta^{-1} \frac{R^2M}{\pi^2 \tildealphamin}, \qquad
  \Gamma \leq \eta \theta^{-1} \invdataweight.
\end{align}

\subsection{Error analysis}

Now we show an estimate for the error $u-\tilde{u}$ in the
$L^2(\Omega)$-norm.

\begin{theorem}
  \label{thm:error_estimate}
  Let $u$ and $\tilde{u}$ be the solutions of \eqref{eq:exact_problem}
  and \eqref{eq:augmented_auxiliary_problem}, respectively.
  Furthermore assume that
  $\tilde{\ell}-\tilde{a}(u, \cdot) \in L^2(\Omega)$.
  Then we have with $\Gamma$ from Lemma~\ref{lem:l2coercivity}
  \begin{align}
    \label{eq:data_pde_error_abstract}
    \| \tilde{u}-u\|_{L^2(\Omega)}
    \leq \sqrt{\Gamma}\| \tilde{u}-u\|_{\tilde{c}}
    \leq \Gamma \Bigl( \Epde + \dataweight \Edata \Bigr)
  \end{align}
  and, using $\theta$ from \eqref{eq:weight_choice},
  \begin{align}
    \label{eq:data_pde_error}
    \| \tilde{u}-u\|_{L^2(\Omega)} \leq
    \sqrt{\Gamma}\| \tilde{u}-u\|_{\tilde{c}} \leq
      \eta \theta^{-1}
        \left( \frac{R^2M}{\pi^2 \tildealphamin}
            \Epde + \Edata \right)
  \end{align}
  with the PDE and data error terms
  \begin{align*}
    \Epde &\colonequals \|\tilde{\ell}-\tilde{a}(u, \cdot)\|_{L^2(\Omega)}, &
    \Edata &\colonequals \|\widetilde{b(u,\cdot)} - b(u,\cdot)\|_{L^2(\Omega)}.
  \end{align*}
\end{theorem}
\begin{proof}
  Testing~\eqref{eq:augmented_auxiliary_problem} with $\tilde{u}-u$ and
    subtracting $\tilde{c}(u,\tilde{u}-u)$ yields
  \begin{align*}
    \|\tilde{u}-u\|_{\tilde{c}}^2
      &= \tilde{r}(\tilde{u}-u) - \tilde{c}(u,\tilde{u}-u) \\
      &= (\tilde{\ell}-\tilde{a}(u, \cdot))(\tilde{u}-u)
        + \dataweight (\widetilde{b(u,\cdot)} - b(u,\cdot))(\tilde{u}-u)\\
      &\leq \Bigl(
        \|\tilde{\ell}-\tilde{a}(u, \cdot)\|_{L^2(\Omega)}
        + \dataweight\|\widetilde{b(u,\cdot)} - b(u,\cdot)\|_{L^2(\Omega)}
        \Bigr) \|\tilde{u}-u\|_{L^2(\Omega)}\\
      &\leq \Bigl(
        \|\tilde{\ell}-\tilde{a}(u, \cdot)\|_{L^2(\Omega)}
        + \dataweight\|\widetilde{b(u,\cdot)} - b(u,\cdot)\|_{L^2(\Omega)}
        \Bigr) \sqrt{\Gamma}\|\tilde{u}-u\|_{\tilde{c}}
  \end{align*}
  where we used the $L^2(\Omega)$-coercivity from
  Lemma~\ref{lem:l2coercivity} for the last estimate.
  Dividing by $\|\tilde{u}-u\|_{\tilde{c}}$ and
  using Lemma~\ref{lem:l2coercivity} again to bound the left hand side,
  we obtain \eqref{eq:data_pde_error_abstract}.
  Estimate \eqref{eq:data_pde_error} is obtained
  by using the bounds on $\Gamma$ from \eqref{eq:weight_constant_bounds}.
\end{proof}

It should be noted that the residual $\tilde{\ell} - \tilde{a}(u,\cdot)$ is zero
if the PDE is exact. Hence $\Epde$ quantifies the error induced by the inexact PDE.
If the data points provide a more fine grained covering of $\Omega$, then $R$
is decreased. In this sense, the PDE error can be reduced by adding more data.
On the other hand simply adding more data cannot cure the data error $\Edata$
made by using inexact data.

The estimate also indicates that the PDE and data term should be balanced appropriately:
In order minimize the constant $\Gamma$ in front of the PDE error term, $\invdataweight$
should be sufficiently small,
while it should be sufficiently large in order minimize
the constant $\Gamma\dataweight$ in front of the data error term.
Since the estimate \eqref{eq:data_pde_error} bounds both constants
in terms of $\theta^{-1}$ the optimal choice of $\invdataweight$
is \eqref{eq:optimal_data_weight} which leads to $\theta=1$.
This motivates the following definition
which allows to characterize the involved constants.
It has to be understood in the sense that the data term
and selected parameter are part of a sequence of problems.

\begin{definition}
  \label{def:nondegenerate_wellbalanced}
  Problem \eqref{eq:augmented_auxiliary_problem}
  is called \emph{non-degenerate} if
  the decomposition used in the data term is non-degenerate
  in the sense that $\eta,M \in O(1)$.
  It is called \emph{well-balanced}
  if $\invdataweight$ is selected such that there is a $\theta \in (0,1)$
  according to \eqref{eq:weight_choice} with $\theta^{-1} \in O(1)$.
\end{definition}

The following corollary summarizes the result for a 
\emph{non-degenerate} and \emph{well-balanced} problem.

\begin{corollary}
  Let $u$ and $\tilde{u}$ be the solutions of \eqref{eq:exact_problem}
  and \eqref{eq:augmented_auxiliary_problem}, respectively.
  Furthermore assume that $\tilde{\ell}-\tilde{a}(u, \cdot) \in L^2(\Omega)$
  and that the problem is \emph{non-degenerate} and \emph{well-balanced}. Then
  \begin{align*}
    \|\tilde{u} - u\|_{L^2(\Omega)}
      \leq CR^2 \Epde + C \Edata.
  \end{align*}
\end{corollary}

Next we will discuss the PDE error term.

\begin{remark}
  Under assumptions on the smoothness of $\partial \Omega$, $u$,
  and $\tildealpha-\alpha$ it can be shown, that $\tilde{\ell}-\tilde{a}(u, \cdot) \in L^2(\Omega)$
  and furthermore that this term can be bounded with respect to $\tildealpha-\alpha$.
  To show this we first note that for all $v \in V$ we have
  \begin{align*}
    |\tilde{\ell}(v)-\tilde{a}(u, v)|
      &= |\tilde{\ell}(v) - (\ell(v)-a(u,v)) -\tilde{a}(u, v)|\\
      &\leq
      \Bigl(\|\tilde{\ell} - \ell\|_{L^2(\Omega)}
      + \|\tildesigma-\sigma\|_{L^\infty(\Omega)} \|u\|_{L^2(\Omega)}
      \Bigr)\|v\|_{L^2(\Omega)}
        + \left| \int_{\Omega} (\tildealpha-\alpha) \nabla u \cdot \nabla v \, dx \right|.
  \end{align*}
  To estimate the last term we note that for $H_0^1(\Omega) \subset V \subset H^1(\Omega)$,
  problem~\eqref{eq:exact_problem} corresponds to a mixed boundary value
  problem with
  \begin{align*}
    u=0 \text{ on }(\partial \Omega)_D, \qquad
    \frac{\partial u}{\partial\nu}=0 \text{ on }(\partial \Omega)_N
  \end{align*} 
  for a decomposition $\partial \Omega = (\partial \Omega)_D \cup (\partial \Omega)_N$.
  For sufficiently smooth $\partial \Omega$ we then have
  \begin{align}
    \label{eq:boundary_terms}
    \Bigl(\frac{\partial u}{\partial \nu} v\Bigr)|_{\partial \Omega} = 0
    \qquad \forall v \in V.
  \end{align}
  Now,
  assuming that $\tildealpha-\alpha \in W^{1,\infty}(\Omega)$ and
  $u \in H^2(\Omega)$
  we get
  \begin{align*}
    \left|\int_\Omega (\tildealpha - \alpha) \nabla u \cdot \nabla v\, dx \right|
      &= \left|
        -\int_\Omega \op{div}((\tildealpha - \alpha) \nabla u) v\, dx
        + \int_{\partial\Omega} (\tildealpha - \alpha)\Bigl(\frac{\partial u}{\partial \nu} v\Bigr)\, ds  \right| \\
      &= \left|\int_\Omega \Bigl(
        (\tildealpha - \alpha) \Delta u
        +\nabla (\tildealpha - \alpha) \cdot \nabla u\Bigr) v\, dx \right|\\
      &\leq \Bigl(
        \|\tildealpha - \alpha\|_{L^\infty(\Omega)} \|\Delta u\|_{L^2(\Omega)}
        + \|\nabla \tildealpha - \nabla \alpha\|_{L^\infty(\Omega)} \|\nabla u\|_{L^2(\Omega)}
        \Bigr) \|v\|_{L^2(\Omega)} \\
      &\leq
        \|\tildealpha -\alpha\|_{W^{1,\infty}(\Omega)} \|u\|_{H^2(\Omega)}\|v\|_{L^2(\Omega)}
      \qquad \forall v \in V.
  \end{align*}
  Hence we have shown the PDE error bound
  \begin{align*}
    \Epde = \|\tilde{\ell}-\tilde{a}(u, \cdot)\|_{L^2(\Omega)}
    \leq \|\tilde{\ell} - \ell\|_{L^2(\Omega)}
      + \|\tildealpha -\alpha\|_{W^{1,\infty}(\Omega)} \|u\|_{H^2(\Omega)}
      + \|\tildesigma-\sigma\|_{L^\infty(\Omega)} \|u\|_{L^2(\Omega)}.
  \end{align*}
\end{remark}

\begin{remark}
  It can also be shown, that such a bound on $\|a(u,\cdot) -\tilde{a}(u,\cdot)\|_{L^2(\Omega)}$
  does in general not exist if $\tildealpha-\alpha$ is not sufficiently smooth---regardless of
  the smoothness of $u$.
  To see this let $\Omega = (-1,1)$, $(\tildealpha - \alpha)(x) = \beta(x)= -\op{sgn}(x) \in L^{\infty}(\Omega)$,
  and $u \in C^\infty([-1,1])$ such that $u|_{[-0.5,0.5]}(x) = x$.
  Now let $v_{\varepsilon}(x) = \op{max}\{0, 1-{\varepsilon}^{-1}|x|\}$ for $\varepsilon<0.5$.
  Then
  \begin{align*}
    \int_\Omega \beta \nabla u \cdot \nabla v_\varepsilon\, dx
      = \int_{[-\varepsilon,\varepsilon]} (-\op{sgn}(x)) (-\op{sgn}(x) \varepsilon^{-1}) \, dx
      = 2
  \end{align*}
  while $\|v_{\varepsilon}\|_{L^2(\Omega)} \to 0$ for $\varepsilon \to 0$.
\end{remark}

Finally, we will provide an error bound for the data error term $\Edata$
in the special case of point-wise approximations
\begin{align}
  \label{eq:point_data}
  b_i = u(p_i) \approx \FINT_{B_i} u
\end{align}
for points $p_i \in B_i$.
To make sense of this expression, we need at least $u \in C(\Omega)$.
However, to derive an error bound in terms of $\op{diam}(B_i)$
we will also need to relate $u(p_i) - u(x)$ to $p_i - x$ for
points $x \in B_i$, independently of $p_i$.
Hence we need to make additions assumptions on the regularity of $u$.

\begin{proposition}
  Assume $u \in C^{0,\lambda}(\overline{\Omega})$ for $\lambda \in (0,1]$,
  i.e. $u$ is $\lambda$-Hölder continuous or Lipschitz continuous (for $\lambda=1$)
  and that the data $\widetilde{b(u,\cdot)}$ is given
  by~\eqref{eq:data_term} with~\eqref{eq:point_data}.
  Then we can bound the data error term $\Edata$ according to
  \begin{align*}
    \Edata = \|\widetilde{b(u,\cdot)} - b(u,\cdot)\|_{L^2(\Omega)}
    \leq \|u\|_{C^{0,\lambda}(\overline{\Omega})}
      M |\Omega|^{1/2} r^\lambda.
  \end{align*}
\end{proposition}
\begin{proof}
  Let $v \in L^2(\Omega)$ and $i=1,\dots,m$. Then
  \begin{align*}
    |b_i - \FINT_{B_i} u|
    = |B_i|^{-1}\left|\int_{B_i} u(p_i) - u(x) \, dx\right|
    \leq \|u\|_{C^{0,\lambda}(B_i)} \op{diam}(B_i)^\lambda
    \leq \|u\|_{C^{0,\lambda}(\overline{\Omega})} r^\lambda.
  \end{align*}
  Using this estimate and proceeding as in the proof of
  Lemma~\ref{lem:data_form_L2_cont} we get
  \begin{align*}
    |\widetilde{b(u,v)} - b(u,v)|
    &= \left|\sum_{i=1}^m \left(b_i - \FINT_{B_i}u\right) |B_i| \FINT_{B_i}v \right| \\
    &\leq \|u\|_{C^{0,\lambda}(\overline{\Omega})} r^\lambda
    \sum_{i=1}^m |B_i| \left|\FINT_{B_i}1\FINT_{B_i}v \right|
    \leq \|u\|_{C^{0,\lambda}(\overline{\Omega})} r^\lambda
      M |\Omega|^{1/2} \|v\|_{L^1(\Omega)}.
  \end{align*}
\end{proof}

Now we summarize the results including the estimates
for the PDE and data error terms.

\begin{corollary}
  \label{cor:L2_estimate}
  Let $u$ and $\tilde{u}$ be the solutions of \eqref{eq:exact_problem}
  and \eqref{eq:augmented_auxiliary_problem}, respectively.
  Furthermore assume that
  \begin{itemize}
    \item
      the problem is \emph{non-degenerate} and \emph{well-balanced},
    \item
      the data $\widetilde{b(u,\cdot)}$ is given by~\eqref{eq:data_term} and~\eqref{eq:point_data},
    \item
      the exact and inexact coefficients satisfy
      $\tildealpha-\alpha \in W^{1,\infty}(\Omega)$,
    \item
      $\partial \Omega$ is smooth enough such that \eqref{eq:boundary_terms} holds true,
    \item
      the solution satisfies $u \in H^2(\Omega) \cap C^{0,\lambda}(\overline{\Omega})$
      for $\lambda\in(0,1]$.
  \end{itemize}
  Then
  \begin{align*}
    \|\tilde{u} - u\|_{L^2(\Omega)}
      \leq
    CR^2 \bigl(
    \|\tilde{\ell} - \ell\|_{L^2(\Omega)}
    + \|\tildesigma - \sigma\|_{L^\infty(\Omega)}
    +\|\tildealpha -\alpha\|_{W^{1,\infty}(\Omega)}\bigr)
    + C r^\lambda.
  \end{align*}
\end{corollary}

\subsection{Pure data fitting}
As a special case of the setting outlined above, one can consider
the pure data fitting problem, where nothing is known about the
partial differential equation satisfied by $u$. In this case, one
could consider using only the data in terms of a least squares
ansatz
\begin{align}
  \label{eq:data_fitting}
  \tilde{u} \in V: \qquad
  \|\tilde{u}-u\|_b^2 \leq
  \|v-u\|_b^2
  \qquad \forall v \in V,
\end{align}
or, equivalently,
\begin{align*}
  \tilde{u} \in V: \qquad
  b(\tilde{u},v) = b(u,v)
  \qquad \forall v \in V.
\end{align*}
Due to the finite rank of $b(\cdot,\cdot)$ this problem
is in general under determined and thus ill-posed.
As a remedy, this can be considered in a finite dimensional
subspace or submanifold $V_h$ of $V$ only, which may lead to
a well-posed problem. However, this comes at the price,
that the behavior of the fitted solution is largely
determined by $V_h$.

As a simple example consider $V = H^1([0,1])$ and $V_h = \mathcal{P}_{m-1}$.
Then, for small sets $B_i$, the problem essentially reduces
to interpolation problem in $\mathcal{P}_{m-1}$ which may
lead to uncontrollably large errors.

To overcome the ill-posedness in a controllable way,
we introduce the regularized problem
\begin{align}
  \label{eq:regularized_data_fitting}
  \tilde{u} \in V: \qquad
    \|\nabla \tilde{u} \|_{L^2(\Omega)}^2
    + \dataweight\|\tilde{u}-u\|_b^2
  \leq
    \|\nabla v \|_{L^2(\Omega)}^2
    + \dataweight\|v-u\|_b^2
  \qquad \forall v \in V.
\end{align}
Noting that this takes the form of~\eqref{eq:augmented_auxiliary_problem}
with $\tildealpha= 1$ and $\tilde{f}=0$
and exact data $\widetilde{b(u,\cdot)} = b(u,\cdot)$
we can utilize Theorem~\ref{thm:error_estimate} to get an error estimate.

\begin{theorem}
  \label{thm:data_fitting}
  Assume that $V \subset H^1_0(\Omega)$, $u \in V \cap H^2(\Omega)$
  and let $\tilde{u}$ be the solution of~\eqref{eq:augmented_auxiliary_problem}
  with $\tildealpha= 1$ and $\tilde{f}=0$. Then
  \begin{align}
    \| \tilde{u}-u \|_{L^2(\Omega)}
      \leq \eta \theta^{-1}\frac{R^2M}{\pi^2}
      \|\Delta u\|_{L^2(\Omega)}
      + \eta \theta^{-1} \|\widetilde{b(u,\cdot)} - b(u,\cdot)\|_{L^2(\Omega)}.
  \end{align}
\end{theorem}
\begin{proof}
  We only need to note that $u \in H^2(\Omega)$
  solves the PDE~\eqref{eq:exact_problem}
  with $f = -\Delta u \in L^2(\Omega)$
  and $\alpha=1$.
  Then Theorem~\ref{thm:error_estimate} provides the error estimate with
  $\tilde{\ell}-\tilde{a}(u, \cdot) = 0 - a(u,\cdot) = -f = \Delta u$.
\end{proof}

In the case of exact data $\widetilde{b(u,\cdot)} = b(u,\cdot)$
in \eqref{eq:regularized_data_fitting}
the estimate reduces to
\begin{align*}
  \| \tilde{u}-u \|_{L^2(\Omega)}
    \leq \eta \theta^{-1}\frac{R^2M}{\pi^2}
    \|\Delta u\|_{L^2(\Omega)}.
\end{align*}
It can also be shown that the result is quasi-optimal
in a certain sense. To illustrate this, we investigate the
question of how well we can approximate $u$
only from the given data
$\FINT_{B_i} u$, $i=1\dots,m$
and the known regularity $u \in H^1_0(\Omega)$.
Unfortunately, there is an
infinite dimensional affine subspace
\begin{align*}
  V_b = \Bigl\{ v \in H_0^1(\Omega) \st \FINT_{B_i} v = \FINT_{B_i} u, i=1,\dots,b \Bigr\}.
\end{align*}
All functions in this space provide a perfect fit to
the data and cannot be distinguished in terms of the
available information. Hence the best we can afford
in view of the regularity $u \in H_0^1(\Omega)$
is to compute a norm minimizing approximation in $V_b$,
i.e.
\begin{align*}
  u_{\text{opt}} = \op{argmin}_{v \in V_b} \|\nabla v\|_{L^2(\Omega)}^2.
\end{align*}
Using the orthogonality
\begin{align*}
  (\nabla u_{\text{opt}}, \nabla (u_{\text{opt}} -v)) = 0 \qquad \forall v \in V_b
\end{align*}
and the fact that $\|u_{\text{opt}}-u\|_b = 0$
we obtain the identity
\begin{align*}
  \tilde{c}(u_{\text{opt}}-u,u_{\text{opt}}-u) = (-\nabla u, \nabla(u_{\text{opt}}-u))
\end{align*}
with $\tilde{c}(\cdot,\cdot)$ as in Theorem~\ref{thm:data_fitting}.
Now assuming the additional regularity $u \in H^2(\Omega)$ and utilizing
the coercivity from Lemma~\ref{lem:l2coercivity} we obtain
\begin{align*}
  \|u_{\text{opt}}-u\|^2_{L^2(\Omega)}
  \leq \eta\frac{R^2M}{\pi^2 } (\Delta u, u_{\text{opt}}-u)
  \leq \eta\frac{R^2M}{\pi^2 } \|\Delta u\|_{L^2(\Omega)} \|u_{\text{opt}}-u\|_{L^2(\Omega)}.
\end{align*}
Thus---if we chose the optimal weighting parameter $\dataweight$
such that $\theta=1$---the error estimate for $\tilde{u}$ coincides with the one
for $u_{\text{opt}}$ which is the energy minimizing one
among all functions that fit the data. It should be noted
that the only reason for the inequality `$\leq$'
in the estimate for $u_{\text{opt}}$
is the application of the coercivity estimate from
Lemma~\ref{lem:l2coercivity}
and the application of the Cauchy--Schwarz inequality in $L^2(\Omega)$.
Hence we cannot expect to be able to derive a better estimate
unless we sharpen the coercivity bound.


\section{Generalized Galerkin discretization}
\label{sec:discretization}

\subsection{Abstract a priori error estimate}
Now we investigate the discretization of the PDE-regularized
problem \eqref{eq:augmented_auxiliary_problem}.
To this end we consider an abstract generalized
Galerkin discretization
\begin{align}
  \label{eq:augmented_auxiliary_problem_galerkin_min}
  \tilde{u}_h = \op{argmin}_{v\in V_h} \tilde{J}(v),
\end{align}
of the minimization formulation \eqref{eq:augmented_auxiliary_problem_min}
of \eqref{eq:augmented_auxiliary_problem}
in a subset $V_h \subset V$.
We call this problem a \emph{generalized Galerkin discretization}
since we do not require that $V_h$ is a closed subspace of $V$
as in a classical Galerkin discretization.
As a consequence, existence and uniqueness
of $\tilde{u}_h$ is not guaranteed.
However, if $V_h$ is a closed subspace, then
\eqref{eq:augmented_auxiliary_problem_galerkin_min}
is equivalent to the variational equation
\begin{align}
  \label{eq:augmented_auxiliary_problem_galerkin}
  \tilde{u}_h \in V_h:
  \qquad
  \tilde{c}(\tilde{u}_h,v) = \tilde{r}(v)
  \qquad
  \forall v \in V_h,
\end{align}
and uniqueness and existence is guaranteed by the
Lax--Milgram theorem.

For several reasons we cannot directly apply the classical
C\'ea-Lemma or Galerkin orthogonality to derive an error bound:
The set $V_h$ is in general not a subspace,
we want to bound the error in the $L^2(\Omega)$-norm,
and the bilinear form incorporates data dependent weighting factors.
Furthermore we are interested in directly bounding $\|\tilde{u}_h-u\|_{L^2(\Omega)}$
and not just $\|\tilde{u}_h-\tilde{u}\|_{L^2(\Omega)}$ for
the auxiliary continuous solution $\tilde{u}$.
As a remedy we will first use a nonlinear C\'ea-Lemma
in terms of the weighted norm $\sqrt{\Gamma}\|\cdot\|_{\tilde{c}}$
and then go over to the $\|\cdot\|_{L^2(\Omega)}$ norm using
Lemma~\ref{lem:l2coercivity}.
The idea of using generalized versions of C\'ea's lemma
for discretization in subsets goes back to
\cite{GrohsHarderingSander2015,Hardering2015} where this was developed
in a metric space setting. Since our setting is more special,
we give a direct proof here.

\begin{lemma}
  \label{lem:generalized_cea}
  Let $\tilde{u}_h \in V_h$ be a solution of~\eqref{eq:augmented_auxiliary_problem_galerkin_min}
  and $\tilde{u} \in V$ the solution of~\eqref{eq:augmented_auxiliary_problem_min}.
  Then we have
  \begin{align*}
    \|\tilde{u}_h-\tilde{u}\|_{\tilde{c}}
    \leq \inf_{v\in V_h}\|v-\tilde{u}\|_{\tilde{c}}.
  \end{align*}
\end{lemma}
\begin{proof}
  First we note that, in view of \eqref{eq:augmented_auxiliary_problem},
  we have for arbitrary $v\in V$
  \begin{align*}
    \|v-\tilde{u}\|_{\tilde{c}}^2
    = \|v\|_{\tilde{c}}^2 - 2\tilde{c}(\tilde{u},v) + \|\tilde{u}\|_{\tilde{c}}^2
    = \|v\|_{\tilde{c}}^2 - 2\tilde{r}(v) + \|\tilde{u}\|_{\tilde{c}}^2
    = 2 \tilde{J}(v) + \|\tilde{u}\|_{\tilde{c}}^2.
  \end{align*}
  Thus we get for the minimizer $\tilde{u}_h$ of $\tilde{J}$ in $V_h$
  \begin{align*}
    \|\tilde{u}_h-\tilde{u}\|_{\tilde{c}}^2
    = 2 \tilde{J}(\tilde{u}_h) + \|\tilde{u}\|_{\tilde{c}}^2
    \leq 2 \tilde{J}(v) + \|\tilde{u}\|_{\tilde{c}}^2
    = \|v-\tilde{u}\|_{\tilde{c}}^2.
  \end{align*}
\end{proof}

\begin{theorem}
  \label{thm:discretization_estimate}
  Let $\tilde{u}_h \in V_h$ be a solution of~\eqref{eq:augmented_auxiliary_problem_galerkin_min}
  and $u \in V$ the solution of~\eqref{eq:exact_problem}.
  Furthermore assume that $\tilde{\ell}-\tilde{a}(u, \cdot) \in L^2(\Omega)$.
  Then $\tilde{u}_h$ satisfies the error bound
  \begin{align*}
    \|\tilde{u}_h - u\|_{L^2(\Omega)}
      \leq \sqrt{\Gamma} \|\tilde{u}_h - u\|_{\tilde{c}}
      &\leq 2\sqrt{\Gamma} \|\tilde{u} - u\|_{\tilde{c}} + \inf_{v \in V_h} \sqrt{\Gamma}\|v-u\|_{\tilde{c}}\\
      &\leq 2\Gamma \Bigl( \Epde + \dataweight \Edata \Bigr) + \inf_{v \in V_h} \sqrt{\Gamma}\|v-u\|_{\tilde{c}}
  \end{align*}
  with the PDE- and data error terms $\Epde$ and $\Edata$
  as defined in Theorem~\ref{thm:error_estimate}.
\end{theorem}
\begin{proof}
  Using Lemma~\ref{lem:l2coercivity}, the triangle inequality,
  and the generalized C\'ea-Lemma \ref{lem:generalized_cea} we get
  \begin{align*}
    \|\tilde{u}_h - u\|_{L^2(\Omega)}
      &\leq \sqrt{\Gamma} \|\tilde{u}_h - u\|_{\tilde{c}}
      \leq \sqrt{\Gamma} \|\tilde{u}_h - \tilde{u}\|_{\tilde{c}}
        + \sqrt{\Gamma}\|\tilde{u} - u\|_{\tilde{c}}\\
      &\leq \sqrt{\Gamma} \|v - \tilde{u}\|_{\tilde{c}}
        + \sqrt{\Gamma}\|\tilde{u} - u\|_{\tilde{c}}
      \leq \sqrt{\Gamma} \|v - u\|_{\tilde{c}}
        + 2\sqrt{\Gamma}\|\tilde{u} - u\|_{\tilde{c}}
      \qquad \forall v \in V_h.
  \end{align*}
  Now using Theorem~\ref{thm:error_estimate} we get the assertion.
\end{proof}

While the first error term is the same as the one
in the continuous case, we still need to take care for the best
approximation error because it involves the weighted data dependent
norm $\sqrt{\Gamma}\|\cdot\|_{\tilde{c}}$.
The main ingredient is the following bound on this norm.

\begin{lemma}
  \label{lem:c_norm_bound}
  The weighted norm $\sqrt{\Gamma}\|\cdot\|_{\tilde{c}}$ can be bounded according to
  \begin{align*}
    \sqrt{\Gamma} \|v\|_{\tilde{c}}
      \leq \sqrt{\eta\theta^{-1}M}\left(
        \frac{R}{\pi} \frac{\sqrt{\tildealphamax}}{\sqrt{\tildealphamin}} \|\nabla v\|_{L^2(\Omega)}
        + \|v\|_{L^2(\Omega)}
        \right).
  \end{align*}
\end{lemma}

\begin{proof}
  As a direct consequence of the bounds in~\eqref{eq:weight_constant_bounds}
  and Lemma~\ref{lem:data_form_L2_cont} we get
  \begin{align*}
    \Gamma \|v\|_{\tilde{c}}^2
      \leq \eta\theta^{-1}M\left(
        \frac{R^2 \tildealphamax}{\pi^2 \tildealphamin} \|\nabla v\|_{L^2(\Omega)}^2
        + \|v\|^2_{L^2(\Omega)}
        \right).
  \end{align*}
\end{proof}

To interpret this estimate we again consider a
\emph{non-degenerate}, \emph{well-balanced} problem.
In this case the estimate in
Lemma~\ref{lem:c_norm_bound}
takes the form
\begin{align*}
  \sqrt{\Gamma} \|v\|_{\tilde{c}} \leq C \Bigl(R \|\nabla v\|_{L^2(\Omega)} + \|v\|_{L^2(\Omega)} \Bigr)
\end{align*}
where $R$ is bounded and even decreasing if the data points cover the domain better and better.
We summarize the result in the following corollary.

\begin{corollary}
  \label{cor:l2_discretization_estimate}
  Let $\tilde{u}_h \in V_h$ be a solution of~\eqref{eq:augmented_auxiliary_problem_galerkin_min}
  and $u \in V$ the solution of~\eqref{eq:exact_problem}.
  Furthermore assume that $\tilde{\ell}-\tilde{a}(u, \cdot) \in L^2(\Omega)$
  and that the problem is \emph{non-degenerate} and \emph{well-balanced}. Then
  \begin{align*}
    \|\tilde{u}_h - u\|_{L^2(\Omega)}
      \leq C \Eapprox + CR^2 \Epde + C \Edata
  \end{align*}
  with the PDE and data error terms as in Theorem~\ref{thm:error_estimate}
  and the approximation-error term
  \begin{align*}
    \Eapprox &\colonequals \inf_{v\in V_h} \Bigl(R \|\nabla(v-u)\|_{L^2(\Omega)} + \|v-u\|_{L^2(\Omega)} \Bigr).
  \end{align*}
\end{corollary}

Inserting the bounds on the PDE and data error we finally get:

\begin{corollary}
  \label{cor:total_discretization_estimate}
  Let $\tilde{u}_h \in V_h$ be a solution of~\eqref{eq:augmented_auxiliary_problem_galerkin_min}
  and $u \in V$ the solution of~\eqref{eq:exact_problem}.
  Furthermore assume that the assumptions of Corollary~\ref{cor:L2_estimate} hold true.
  Then
  \begin{multline*}
    \|\tilde{u}_h - u\|_{L^2(\Omega)}
      \leq
    CR^2 \bigl(
    \|\tilde{\ell} - \ell\|_{L^2(\Omega)}
    + \|\tildesigma - \sigma\|_{L^\infty(\Omega)}
    +\|\tildealpha -\alpha\|_{W^{1,\infty}(\Omega)}\bigr)\\
    + C r^\lambda
    + C\inf_{v\in V_h} \Bigl(R \|\nabla(v-u)\|_{L^2(\Omega)} + \|v-u\|_{L^2(\Omega)} \Bigr).
  \end{multline*}
\end{corollary}

It is important to note that we neither require that
$V_h$ is a subspace, nor that the solution to
\eqref{eq:augmented_auxiliary_problem_galerkin_min}
is unique.
Hence the error estimate is in principle applicable to
nonlinear approximation schemes.
However, it will in general be hard to compute
a global minimizer in $V_h$ as required
in \eqref{eq:augmented_auxiliary_problem_galerkin_min}
since practical optimization schemes often at most guarantee
local optimality.

In case of inexact or local minimization we can at least bound
the error in terms of the algebraic energy error.
Let $\tilde{\tilde{u}}_h$ an approximation of $\tilde{u}_h$.
Then using a straight forward modification of the proof
of Lemma~\ref{lem:generalized_cea} we get
\begin{align*}
  \|\tilde{\tilde{u}}_h-\tilde{u}\|_{\tilde{c}}
  \leq \inf_{v\in V_h}\|v-\tilde{u}\|_{\tilde{c}}
  + \sqrt{2(\tilde{J}(\tilde{\tilde{u}}_h) - \tilde{J}(\tilde{u}_h))}.
\end{align*}
Thus, if we want to bound
$\|\tilde{\tilde{u}}_h -u\|_{L^2(\Omega)}$ instead of
$\|\tilde{u}_h -u\|_{L^2(\Omega)}$
we have to add the algebraic error term
\begin{align*}
  \sqrt{2\Gamma(\tilde{J}(\tilde{\tilde{u}}_h) - \tilde{J}(\tilde{u}_h))}
  \leq C R \sqrt{\tilde{J}(\tilde{\tilde{u}}_h) - \tilde{J}(\tilde{u}_h)}
\end{align*}
to the right hand sides of the estimates in
Theorem~\ref{thm:discretization_estimate},
Corollary~\ref{cor:l2_discretization_estimate}, and
Corollary~\ref{cor:total_discretization_estimate}.

\subsection{Finite element discretization}

As an example for a linear Galerkin discretization we apply
the abstract error bound to a finite element ansatz.
To this end let $V_h \subset V$ be a conforming Lagrange finite
element space of order $k$
(piecewise polynomials in $\mathcal{P}_k$ for simplex elements
or piecewise tensor-polynomials in $\mathcal{Q}_k$ for cubic elements)
on a triangulation with mesh size $h$.
Then we can apply the classical finite element
interpolation error estimate which,
in the present special case, provides:

\begin{proposition}
  Let $\Pi_h : C^0(\overline{\Omega}) \to V_h$ interpolation operator
  and $u \in H^{k+1}(\Omega)\cap C^0(\overline{\Omega})$.
  Then
  \begin{align*}
    \|u-\Pi_hu\|_{L^2(\Omega)}
      &\leq C h^{k+1} |u|_{H^{k+1}(\Omega)}, &
    \|u-\Pi_hu\|_{H^1(\Omega)}
      &\leq C h^{k} |u|_{H^{k+1}(\Omega)}.
  \end{align*}
  with a constant $C$ depending only on the shape
  regularity of the triangulation.
  Here $|\cdot|_{H^{k+1}(\Omega)}$ denotes the
  $H^{k+1}(\Omega)$-semi-norm
  containing only derivatives of order $k+1$.
\end{proposition}

For a proof we refer to \cite[Theorem~3.2.1 and Remark~3.2.2]{Ciarlet1978}.
Inserting the interpolation error estimate in the
abstract error bound we get the following total error bound.

\begin{corollary}
  Let $\tilde{u}_h \in V_h$ be a solution of~\eqref{eq:augmented_auxiliary_problem_galerkin_min}
  and $u \in V$ the solution of~\eqref{eq:exact_problem}.
  Furthermore assume that the assumptions of Corollary~\ref{cor:L2_estimate} and $u \in H^{k+1}(\Omega)$ hold true.
  Then
  \begin{multline*}
    \|\tilde{u}_h - u\|_{L^2(\Omega)}
      \leq
    CR^2 \bigl(
    \|\tilde{\ell} - \ell\|_{L^2(\Omega)}
    + \|\tildesigma - \sigma\|_{L^\infty(\Omega)}
    +\|\tildealpha -\alpha\|_{W^{1,\infty}(\Omega)}\bigr)
    + C r^\lambda
    + C Rh^{k} + Ch^{k+1}.
  \end{multline*}
\end{corollary}

\subsection{A heuristic strategy for parameter selection}
\label{subsec:parameter_strategy}

To guarantee \emph{non-degenerate} and \emph{well-balanced}
problem in the sense of Definition~\ref{def:nondegenerate_wellbalanced}
the size of the averaging sets $B_i$ and the parameter
$\invdataweight$ should be carefully selected
such that the quotient $|K_i|/|B_i|$ is bounded
and such that $\invdataweight$ scales like $R^2 M$.
In applications, where data points $p_1,\dots,p_m \in \Omega$
are given we can in general not assume that one is able to
compute a decomposition into sets $K_i$ and their maximal diameter $R$ exactly.
While this is in principle possible using a Voronoi
decomposition, computing the latter is in general
much to costly.
Hence we propose a heuristic strategy to determine the
sets $B_i$ and estimates for $R$ and $\invdataweight$
under the assumption of uniformly distributed data points $p_i$.

To this end we make the assumption that $\Omega$ fits
into a cube $[0,D]^d$.
Then we can expected that the points are in average uniformly spaced.
For an ideal uniformly spaced distribution of $m$ points
in $[0,D]^d$
we can place $m$ non-overlapping cubes $\hat{K}_i$
with edge length $\hat{L}=Dm^{-1/d}$
on a uniform lattice into $\Omega=[0,D]^d$.
The diameter of these boxes is
\begin{align*}
  \hat{R} \colonequals \hat{L}\sqrt{d} = D m^{-1/d} \sqrt{d}
\end{align*}
which we will use as estimate for $R$.
In order to fix $|\hat{K}_i|/|B_i|=Q$ for a parameter $Q > 1$,
we will use cubes $B_i$ centered at $p_i$ with edge length $\hat{l}$
\begin{align*}
  B_i \colonequals p_i + \hat{l}[-0.5,0.5]^d,
  \qquad
  \hat{l}
    \colonequals \hat{L}Q^{-1/d}
    = D(mQ)^{-1/d}.
\end{align*}
Since we cannot control $M$ explicitly, we assume $M \in O(1)$ and
select
\begin{align*}
  \invdataweight
    \colonequals \frac{\hat{R}^2}{\pi^2\tildealphamin}
    = \frac{D^2 m^{-2/d} d}{\pi^2\tildealphamin}.
\end{align*}

\section{Numerical experiments}
\label{sec:numerical_results}

\subsection{A smooth test problem}
Finally we illustrate the theoretical findings using numerical
experiments for test problems. To this end we consider an example
problem from~\cite{EYu2018} given by
\begin{align*}
  -\Delta u + \sigma u &= f \qquad \text{in }\Omega, &
  \frac{\partial u}{\partial n} &=0 \qquad \text{on }\Omega
\end{align*}
with $\Omega=[0,1]^d$.
For the Eigenfunction $u(x) = \sum_{i=1}^d \cos(\pi x_i)$
of the Laplacian and $\sigma=\pi^2$ this equation is
satisfied with $f(x) = (\pi^2+\sigma)u(x) = 2\pi^2 u(x)$.
We do not want to hide, that this example was selected, because
---in contrast to Dirichlet boundary conditions---%
natural boundary conditions $\frac{\partial u}{\partial n} =0$
do not require any approximation for neural network discretizations
which was not covered in the error analysis.

As inexact auxiliary PDE we will use the same differential
operator $-\Delta + \sigma$ but a scaled right hand side
$\tilde{f} = (1-\varepsilon) f$. Then we can explicitly compute the
PDE error term
\begin{align*}
  \Epde
  = \|f-\tilde{f}\|_{L^2(\Omega)}
  = \varepsilon\|f\|_{L^2(\Omega)}.
\end{align*}

In all experiments we generate data by creating $m$ uniformly
distributed random points $p_1,\dots,p_m$
and use cubes $B_i$ with edge length $\hat{l}$ centered at
the points $p_i$.
The points are sampled in $[\hat{l}/2, 1-\hat{l}/2]^d$
such that $B_i \subset \Omega$ is guaranteed.
The edge length $\hat{l}$ and the penalty parameter $\delta$
are selected according to the heuristic strategy
proposed in Subsection~\ref{subsec:parameter_strategy}
for different values of $Q$.
Notice that the heuristic strategy does not guarantee
a \emph{non-degenerate} and \emph{well-balanced} problem
in the sense of Definition~\ref{def:nondegenerate_wellbalanced}.
Nevertheless, the method is still covered
by the theory presented above, but the constants may
blow up if the strategy fails, leading to different error rates.

\subsection{Finite element discretization}

In this section we discretize the problem with conforming
finite elements for $d=1,2,3$.
Since we are not interested in illustrating
well known finite element error bounds,
we use a fixed finite element grid with uniformly
spaced interval/rectangular/hexahedral elements and tensorial
$\mathcal{Q}_k$ Lagrange finite elements of order $k$
throughout the experiments.
For $d=1$ we used $64$ elements and order $k=4$,
for $d=2$ we used $4069 = 64 \times 64$ elements and order $k=4$,
and for $d=3$ we used $4069 = 16 \times 16 \times 16$ elements and order $k=2$.
In any case the discretization error can be neglected
compared to the PDE and data error terms reported
in the following.
All reported $L^2(\Omega)$-errors are computed using a quadrature
rule of order $k^2$ on the grid elements.

For the first experiment we fix $\tilde{f} = 0.5 f$
and compute the solution for increasing number of data
points $m$.
The computations are done using exact data, i.e.,
$b_i = \FINT_{B_i}u$ and inexact data in the form
of point values $b_i = f(p_i)$.
Figure~\ref{fig:fem_error_over_datapoints} depicts
the error over $m$
for dimensions $d=1,2,3$ and parameters $Q=4,2$
in the heuristic strategy.
Left and right picture show exact and inexact data, respectively.
We observe that the error decays like
$O(m^{-2/d})= O(\hat{R}^2)$ as expected from
the theoretical error bounds for exact data
and uniformly distributed data points.
We also find that the error increases by a constant
if we increase $Q=|\hat{K}_i|/|B_i|$
which is again in accordance with the error estimate where
this enters in the constant factor $\eta$.
For inexact data, the situation is very similar.
We observe the order $O(m^{-2/d})= O(\hat{R}^2)$
but for very small errors obtained for $d=1$
eventually the case $Q=4$ becomes better compared
to $Q=2$.
This can be explained by the fact that, while increasing
$Q$ increases the constant in the PDE error term, it also
decreases $r$ and thus the data error term.
Hence, if the data error comes in the range of the
PDE error, we expect that larger $Q$ reduces the total error.

In the second experiment we fix $m=512$ data points
and compute the solution for inexact right hand sides
$\tilde{f} = (1-\varepsilon) f$ with varying $\varepsilon$
and exact data $b_i = \FINT_{B_i}u$ as well as inexact
data $b_i = f(p_i)$.
Figure~\ref{fig:fem_error_over_pde} depicts
the error over the PDE error $\varepsilon=\Epde/\|f\|_{L^2(\Omega)}$
for dimensions $d=1,2,3$ and parameters $Q=4,2$
in the heuristic strategy.
Left and right picture show exact and inexact data, respectively.
Here we observe that the error scales like
$O(\varepsilon)=O(\Epde)$ for a wide range of $\varepsilon$.
This is expected from the theoretical error bound.
Again, we observe that for small total error $Q=4$ is better
compared to $Q=2$ where the error finally saturates for $d=2,3$.
The latter indicates that the data error
starts to dominate in this regime such that we no longer benefit
from improving the PDE error.

\begin{figure}
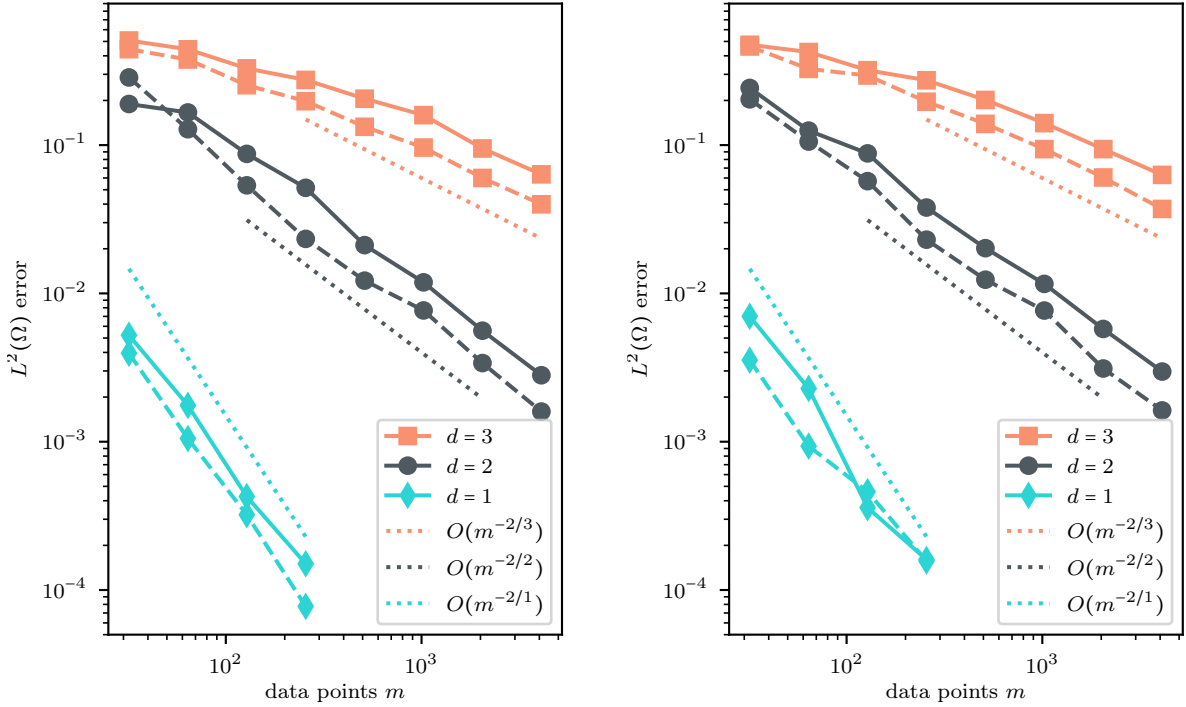

  \centering%
  \input{FEMerrorOverDataAverages.pgf}%
  \hfill
  \input{FEMerrorOverDataPointValues.pgf}%
  \caption{%
    Finite element discretization.
    Total error over number of data points.
    Solid lines: $Q=4$. Dashed lines: $Q=2$. Dotted lines: Reference slope.
    Left: Exact local average data.
    Right: Point values as inexact data.
  }
  \label{fig:fem_error_over_datapoints}%
\end{figure}

\begin{figure}
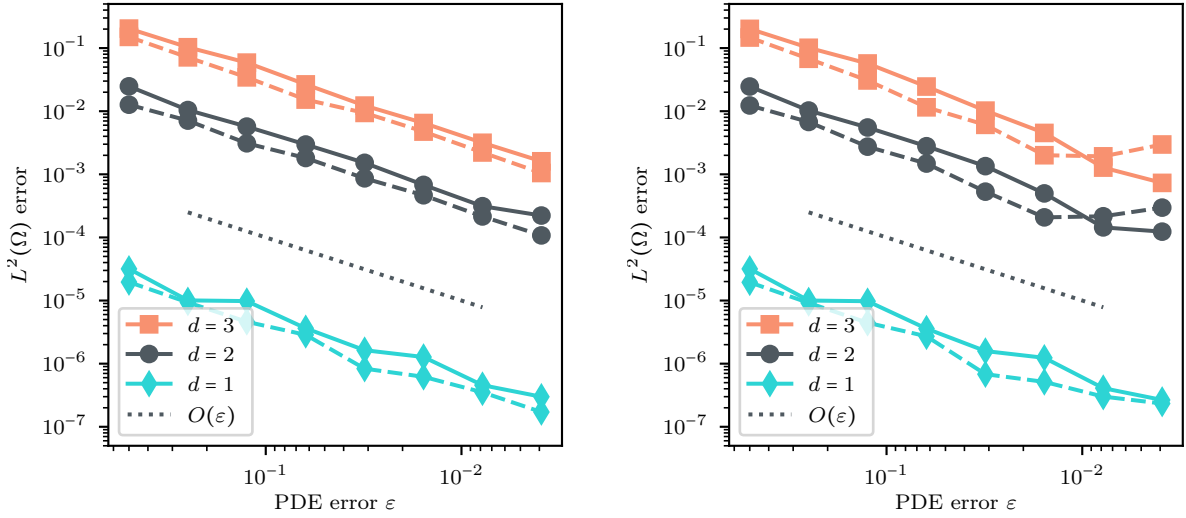

  \centering%
  \input{FEMerrorOverPDEAverages.pgf}%
  \hfill
  \input{FEMerrorOverPDEPointValues.pgf}%
  \caption{%
    Finite element discretization.
    Total error over PDE error for exact local average data.
    Solid lines: $Q=4$. Dashed lines: $Q=2$. Dotted lines: Reference slope.
    Left: Exact local average data.
    Right: Point values as inexact data.
  }
  \label{fig:fem_error_over_pde}%
\end{figure}

\subsection{Neural network discretization}

Finally we consider the discretization of the test problem
with neural networks. To this end we minimize the loss
functional $\tilde{J}$ defined in~\eqref{eq:functional}
over a set $V_h$ of neural networks with a fixed architecture.
In this case integrals are no longer
evaluated exactly, but approximated using stochastic
integration as proposed in~\cite{EYu2018}.
More precisely, we approximate the local averages $\FINT_{B_i} v$
appearing in the functional by averaging $v$ over $10d$
uniformly distributed sampling points in $B_i$.
Since all sets $B_i$ are the same up to translation,
we also translated the sampling points. The integration
over $\Omega$ was approximated using $100d$ uniformly distributed
sampling points in $\Omega$ that are newly sampled in each step
of the iterative algebraic solution method.
Since we use natural boundary conditions, we only need to
approximate the space $H^1(\Omega)$ without any hard boundary
constraints. Hence no approximation of boundary conditions
was needed.

The network architecture is as follows:
A first layer inflates the input size to $16$.
This is followed by $3$ blocks, each consisting of two densely
connected layers with a residual connection.
A final affine layer reduces the size from $16$ to $1$.
In total the network takes the form
\begin{align*}
  F = F_\text{out} \circ F_3 \circ F_2 \circ F_1 \circ F_\text{in}
\end{align*}
where $F_\text{in} : \R^d \to \R^{16}$ is a linear map padding its
input by $16-d$ zeros, $F_\text{out} : \R^{16} \to \R$ is an affine map,
and each block $F_i : \R^{16} \to \R^{16}$ has the form
\begin{align*}
  F_i (x) = \psi(W_{i,2} \psi(W_{i,1} x - \theta_{i,1}) - \theta_{i,2})+x
\end{align*}
for dense weight matrices $W_{i,j} \in \R^{16\times 16}$ and bias vectors
$\theta_{i,j} \in \R^{16}$.
In all layers we used the activation function $\psi(x) = \max\{x^3,0\}$.
All parameters involved in $F_1, F_2,F_3,F_\text{out}$ are determined
in the training procedure.

The networks are trained using the Adam method~\cite{KingmaBa2014}
which is a variant of stochastic gradient descent.
Training was stopped if the
$L^2(\Omega)$ error $\|\tilde{u}_h-u\|_{L^2(\Omega)}$
did not decrease any more. While this is in general impractical
due to the unknown solution $u$, we used this here to avoid
effects of more heuristic stopping criteria.

Again we fix $\tilde{f} = 0.5 f$
and compute the solution for increasing number of data
points $m$.
The computations are done using exact data, i.e.,
$b_i = \FINT_{B_i}u$ and inexact data in the form
of point values $b_i = f(p_i)$.
Figure~\ref{fig:nn_error_over_datapoints} depicts
the error over $m$
for dimensions $d=1,2,3$ and parameters $Q=4,2$
in the heuristic strategy.
Left and right picture show exact and inexact data, respectively.
For $d=2$ and $d=3$ we again observe that the error decays like
$O(m^{-2/d})= O(\hat{R}^2)$.
For $d=1$ the situation is less clear.
While we roughly observe the order $O(m^{-2/d})= O(\hat{R}^2)$ again,
there are some exceptions. Most importantly the error increases
after adding more data in one case.

\begin{figure}
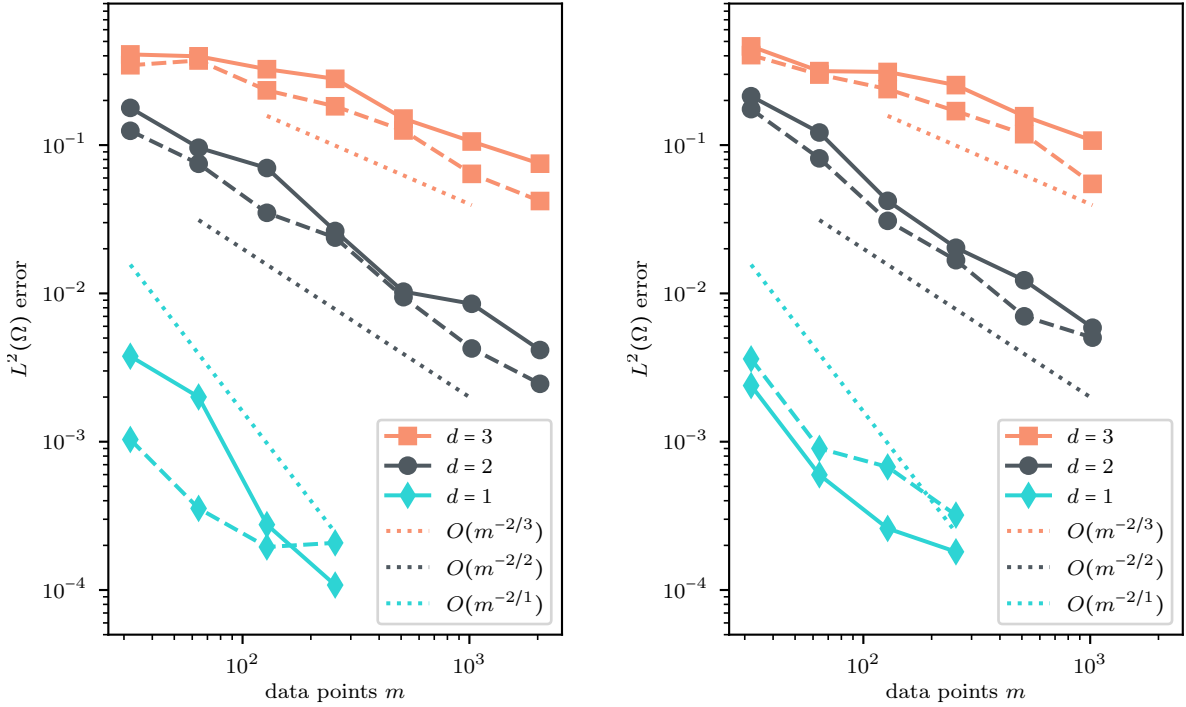

  \centering%
  \input{NNerrorOverDataAverages.pgf}%
  \hfill
  \input{NNerrorOverDataPointValues.pgf}%
  \caption{%
    Neural network discretization.
    Total error over number of data points.
    Solid lines: $Q=4$. Dashed lines: $Q=2$. Dotted lines: Reference slope.
    Left: Exact local average data.
    Right: Point values as inexact data.
  }
  \label{fig:nn_error_over_datapoints}%
\end{figure}

\bibliography{paper}
\bibliographystyle{plain}

\end{document}